\documentclass[12pt]{amsart}
\usepackage{amstext,amsfonts,amssymb,amscd,amsbsy,amsmath,verbatim}
\usepackage[alphabetic,backrefs,lite]{amsrefs} 
\usepackage{ifthen, fullpage, extarrows}
\usepackage{color,tikz}
\usepackage{amsthm}
\usepackage{latexsym}
\usepackage[all]{xy}
\usepackage{enumerate}

\newtheorem{lemma}{Lemma}[section]

\newtheorem{prop}[lemma]{Proposition}
\newtheorem{question}[lemma]{Question}
\newtheorem{cor}[lemma]{Corollary}

\newtheorem{claim*}{Claim}
\newtheorem{thm}[lemma]{Theorem}

\theoremstyle{remark}

\newcommand{\Def}{\operatorname{Def}}

\newcommand{\im}{\operatorname{im}}

\newcommand{\Hom}{\operatorname{Hom}} 
\newcommand{\Ext}{\operatorname{Ext}} 

\newcommand{\kk}{\Bbbk}

\newcommand{\PP}{\mathbb{P}}
\renewcommand{\AA}{\mathbb{A}}
\newcommand{\GG}{\mathbb{G}}

\newcommand{\ZZ}{\mathbb{Z}}

\newcommand{\cO}{\mathcal{O}}

\newcommand{\defi}[1]{\textsf{#1}} 

\def\reg{\operatorname{reg}}

\begin{document}
\title{Murphy's Law for Hilbert function strata in the Hilbert scheme of points}
\author{Daniel Erman}
\date{\today}
\maketitle
\begin{abstract}
An open question is whether the Hilbert scheme of points of a high dimensional affine space satisfies Murphy's Law, as formulated by Vakil.  In this short note, we instead consider the loci in the Hilbert scheme parametrizing punctual schemes with a given Hilbert function, and we show that these loci satisfy Murphy's Law.  We also prove a related result for equivariant deformations of curve singularities with $\mathbb G_m$-action.
\end{abstract}

\section{Introduction}
It remains wide open whether the Hilbert scheme of points in $\AA^n$ satisfies Murphy's Law, in the precise sense introduced in~\cite[p. 2]{vakil}.  Vakil's version of Murphy's Law is based on the notion of \defi{smooth equivalence}, which is the equivalence on local rings generated by the relations $R\sim S$ whenever there exists a smooth map of local rings $R\to S$.  We say that $R$ is a \defi{singularity of finite type over $\mathbb Z$} if $R$ is the completion of a ring of finite type over $\mathbb Z$ at some prime ideal.   motivating question

\begin{question}\label{question}
Up to smooth equivalence, does any singularity of finite type over $\ZZ$ occur on some Hilbert scheme of points in $\AA^n$?
\end{question}
\noindent In fact, beyond Iarrobino's foundational reducibility results~\cite{iarrobino-red}, little is known about the singularities that occur on Hilbert schemes of points in $\AA^n$.

Inside of the Hilbert scheme of points in $\AA^n$, there are various loci parametrizing the homogeneous ideals with a given Hilbert function.  These strata, or variants thereof (including Gr\"obner strata, equivariant Hilbert schemes, etc.) have arisen in previous work on Hilbert schemes~\cite{cevv,evain,lederer}.  We refer to these refined parameter spaces as \defi{punctual Hilbert function strata}, and our main result is a Murphy's Law for these strata.    
\begin{thm}\label{thm}
Up to smooth equivalence, any singularity of finite type over $\ZZ$ occurs on some punctual Hilbert function strata.
\end{thm}

The idea behind this theorem is as follows.  We fix a surface $X\subseteq \mathbb P^{n-1}$ with a pathological deformation space as constructed in~\cite[\S 4]{vakil}.  We then consider the affine cone $CX\subseteq \mathbb A^{n}$ over $X$ and define a $0$-dimensional scheme $\Gamma$ by taking a large infinitesimal neighborhood of the cone point of $CX$.  In the language of ideals, we set $I_\Gamma:=I_X+\mathfrak m^m$, where $I_X$ is the ideal of $CX$, $\mathfrak m$ is the ideal of the cone point, and $m$ is much larger than the Castelnuovo-Mumford regularity of $X$.  Lastly, we use a syzygetic argument to relate the $\GG_m$-equivariant deformations of $\Gamma$ with embedded deformations of $X\subseteq \mathbb P^{n-1}$, obtaining the main result.

Beyond illustrating the unbounded pathologies of these strata, this theorem also suggests---but does not imply---an affirmative answer to Question~\ref{question}.  

Our techniques also provide a closely related result for the equivariant deformation spaces of curve singularities with a $\GG_m$-action.
\begin{cor}\label{cor:curves}
Up to smooth equivalence, any singularity of finite type over $\ZZ$ occurs on the $\GG_m$-equivariant deformation space of some curve singularity with a $\GG_m$-action.
\end{cor}
The paper is organized as follows.  
In \S\ref{sec:defns}, we review the deformation functors that will be used throughout.  In \S\ref{sec:compthms}, we prove Theorem~\ref{thm}, and in \S\ref{sec:examples} we prove Corollary~\ref{cor:curves}.
\section{Review of Deformation Functors}\label{sec:defns}
Throughout, $\kk$ denotes an arbitrary field.  We mostly follow the notation of \cite[\S2.4]{sernesi}.  For an algebraic scheme $Y$, we use $\Def_Y$ to denote the functor from Artin rings to sets which sends an Artinian ring $A$ to the set of deformations of $Y$ over $A$.
Similarly, for a closed subscheme $Y\subseteq Z$ we use $\Def_{Y|Z}$ to denote the functor of embedded deformations. 

If we consider $\mathbb A^{n}_{\kk}$ together with the $\GG_m$-action of dilation, then this leads to a special case of the multigraded Hilbert scheme~\cite{haiman-sturmfels} where we give the polynomial ring $S=\kk[x_1, \dots, x_n]$ the standard grading, and where we study families of graded ideals with a fixed Hilbert function $h$.  If $X\subseteq \mathbb A^n$ is $\GG_m$-invariant, then we use $\Def^{\GG_m}_{X|\mathbb A^{n}}$ to stand for the functor of $\GG_m$-invariant embedded deformations of $X$, and we use $\mathcal H_X$ for the local ring at the point of the multigraded Hilbert scheme defined by $X\subseteq \mathbb A^n$.

With this notation, Theorem~\ref{thm} amounts to the claim that, for any singularity $R$ of finite type of $\ZZ$, there exists an $n$ and a graded ideal $I_{\Gamma}\subseteq \kk[x_1, \dots, x_n]$ such that $\kk[x_1,\dots,x_n]/I_{\Gamma}$ is $0$-dimensional and such that $\mathcal H_{\Gamma}$ is in the same smooth equivalence class as $R$.

\section{Punctual Hilbert function strata}\label{sec:compthms}
We begin with a result comparing the embedded deformation theory of $Y\subseteq \mathbb P^{n-1}$ with the deformation functors that will be essential to the proof of our main result.  Throughout this section, we set $S:=\kk[x_1, \dots, x_n]$ and $\mathfrak m=\langle x_1, \dots, x_n\rangle\subseteq S$.
\begin{prop}\label{prop:muisomorphism}
Let $Y\subseteq \mathbb P^{n-1}$ be a projective scheme defined by the ideal $I_Y\subseteq S$. Let $\Gamma\subseteq \mathbb A^n$ be the $0$-scheme defined by by $I_\Gamma=I_{Y}+\mathfrak m^m$ for any $m\geq \reg(I_Y)+2$.
The functors $\Def^{\GG_m}_{CY|\mathbb A^{n}}$ and $\Def^{\GG_m}_{\Gamma | \mathbb A^{n}}$ are isomorphic.
\end{prop}

\begin{proof}
Consider a minimal free resolution of $I_Y$:
\[
\xymatrix{
\dots \ar[r]&F_3\ar[r]^-{\sigma_3}&F_2\ar[r]^-{\sigma_2}&F_1\ar[r]^-{\sigma_1}&I_Y\ar[r]&0.
}
\]
Since $m>\reg(I_X)$, the minimal free resolution of $I_\Gamma=I_Y+\mathfrak m^m$ has the form
\[
\xymatrix{
\dots\ar[r]&
F_3\oplus S(-m-2)^{t_3} \ar[r]^-{\left(\begin{smallmatrix}\sigma_3&\tau_3\\ 0&\tau_3' \end{smallmatrix}\right)}&
F_2\oplus S(-m-1)^{t_2}\ar[r]^-{\left(\begin{smallmatrix}\sigma_2&\tau_2\\ 0&\tau_2' \end{smallmatrix}\right)}&
F_1\oplus S(-m)^{t_1}\ar[r]^-{\left(\begin{smallmatrix}f & g \end{smallmatrix}\right)}&
I_\Gamma \ar[r]& 0,
}
\]
for some positive integers $t_i$.
The Hilbert function of $I_\Gamma$ is given by:
\[
h_d(I_\Gamma)=
\begin{cases}
h_d(I_Y) & \text{ if } d< m\\
0 & \text{ if } d\geq m
\end{cases}.
\]

We now construct a functorial map: $\Def^{\GG_m}_{Y|\mathbb A^{n}}\to \Def^{\GG_m}_{\Gamma|\mathbb A^{n}}$.  For an Artin ring $A$, every element of $\Def^{\GG_m}_{Y|\mathbb A^{n}}(A)$ is represented by an ideal $I_Y'\subseteq A[x_1, \dots, x_n]$, and our functorial map is given by sending the deformed ideal $I_Y'$ to $I_Y'+\mathfrak m^m$.  (By a minor abuse of notation, we also use $\mathfrak m$ to refer to $\mathfrak mA)$. To check that this map is well defined, we must show that if $I_Y'$ is a $\GG_m$-invariant deformation of $I_Y$ over $A$ then  $I_Y'+\mathfrak m^m$ is a $\GG_m$-invariant deformation of $I_\Gamma$ over $A$.  By the definition of the multigraded Hilbert scheme~\cite[p. 1]{haiman-sturmfels}, this amounts to checking that for every $d\in \mathbb N$, the quotient
\[
(A\otimes S_d)/(I_Y'+\mathfrak m^m)_d
 \]
is locally free of rank $h_d(I_\Gamma)$.  This follows immediately from the observation that
\[
(A\otimes S_d)/(I_Y'+\mathfrak m^m)_d\cong 
\begin{cases}
(A\otimes S_d)/(I_Y')_d & \text{ if } d<m\\
0 & \text{ if } d\geq m
\end{cases}.
\]

To complete the proof of the proposition, it is now sufficient to prove that this map of functors induces an isomorphism on first order deformations and an injection on obstruction spaces (cf. ~\cite[Proof of Theorem~5.1]{pinkham}).
The induced map on first order deformations is given by the composition of the natural maps:
\begin{equation}\label{eqn:tangentspaces}
\Hom_S(I_Y,S/I_Y)_{0}\to \Hom_S(I_Y,S/(I_\Gamma))_{0}\to \Hom_S(I_\Gamma,S/(I_\Gamma))_{0}.
\end{equation}
An element of $\Hom_S(I_Y,S/I_Y)_{0}$ is equivalent to a degree $0$ map $\widetilde{\alpha}\colon F_1\to S/I_Y$ such that $\widetilde{\alpha}\circ \sigma_2=0$.  Each generator of $F_2$ has degree $\leq \reg(I_Y)+1$, and hence the image of $\widetilde{\alpha}\circ \sigma_2$ lands in $(S/I_Y)_{\leq \reg(I_\Gamma)+1}$.  Let $q\colon S/I_Y\to S/I_\Gamma$ be the quotient map.  Since $m\geq \reg(I_\Gamma)+2$, it follows that
$\widetilde{\alpha}\circ \sigma_2\colon F_1\to S/I_Y$ is the zero map if and only if $q\circ \widetilde{\alpha} \circ \sigma_2\colon F_1\to S/I_\Gamma$ is the zero map.  Hence, an equivariant deformation of $Y$, i.e. an element of $\Hom_S(I_Y,S/I_Y)_{0}$, is given by a degree $0$ map $\alpha\colon F_1\to S/I_\Gamma$ such that $\alpha\circ \sigma_2=0$.

An equivariant deformation of $\Gamma$, i.e. an element of $\Hom_S(I_\Gamma,S/(I_\Gamma))_{0}$, is given by a pair of degree $0$ morphisms $(\alpha, \alpha')$ where $\alpha\colon F_1\to S/I_{\Gamma}, \alpha\colon S(-m)^{t_1} \to S/I_\Gamma$ and where:
\begin{equation}\label{eqn:aa'}
\alpha\circ \sigma_2=0, \qquad \text{ and } \qquad  \alpha\circ \tau_2+\alpha'\circ \tau_2'=0.
\end{equation}
However, since $\im(\alpha')\subseteq (S/I_{\Gamma})_m=0$, it follows that $\alpha'$ is actually the zero map. Further, $\alpha\circ \tau_2$ lands in $(S/I_\Gamma)_{m}=0$.  Thus, the second condition above is trivially satisfied for any degree $0$ map.  We conclude that the composition in \eqref{eqn:tangentspaces}, which sends $\alpha \mapsto (\alpha, 0)$, is a bijection.

The induced map of obstruction spaces is given by the composition:
\begin{equation}\label{eqn:obstructionspaces}
\Ext^1_S(I_Y,S/I_Y)_{0}\to \Ext^1_S(I_Y,S/(I_\Gamma))_{0}\to \Ext^1_S(I_\Gamma,S/(I_\Gamma))_{0}.
\end{equation}
We must check injectivity.  A cycle for $\Ext^1_S(I_Y,S/I_Y)_{0}$ may be represented as a map: $\widetilde{\beta}\colon F_2\to S/I_Y$ such that $\widetilde{\beta}\circ \sigma_3=0$.  The map of obstruction spaces sends $\widetilde{\beta}$ to the cycle $(\beta,0)\colon F_2\oplus S(-m-1)^{t_1} \to S/I_\Gamma$ where $\beta:=q\circ \widetilde{\beta}$.  Since $\beta$ has degree $0$, the cycle $(\beta,0)$ is a boundary if and only if there exists some $(\alpha, 0)\colon F_1\oplus S(-m)\to S/I_\Gamma$ such that:
\[
\alpha \circ \sigma_2=\beta \qquad \text{ and }\qquad  \alpha \circ \tau_2 = 0.
\]
Again by degree considerations, the second condition is automatically satisfied.  Let $\widetilde{\alpha}$ any map $\widetilde{\alpha}\colon F_1\to S/I_Y$ such that $q\circ \widetilde{\alpha}=\alpha$.  Since the image of $\alpha \circ \sigma_2$ lands in $(S/I_{\Gamma})_{\leq \reg(I_Y)+1}$, we conclude that
\[
\alpha \circ \sigma_2=\beta \iff \widetilde{\alpha} \circ \sigma_2 = \widetilde{\beta}.
\]
This shows injectivity of the composition in \eqref{eqn:obstructionspaces}, completing the proof.
\end{proof}

We now prove Theorem~\ref{thm}.
\begin{proof}[Proof of Theorem~\ref{thm}]
As in \cite[Theorems 4.4, 4.5]{vakil}, we fix a smooth surface $X$ whose abstract deformation space has the desired singularity type $R$ and such that $\omega_X$ is very ample and $H^1(X,\cO_X)=0$. We may further choose a very ample $\cO_X(1):=\omega_X^{\otimes e}$ for some large $e\gg 0$ such that each of the following holds (see~\cite[\S6.7]{vakil}):
\begin{enumerate}
    \item[(i)] $\oplus_{\nu\in \mathbb Z} H^1(X,\mathcal O_X(\nu))=0$; 
    \item[(ii)] $\oplus_{\nu \ne 0} H^1(X,T_X(\nu))=0$;
    \item[(iii)] The embedding of $X\subseteq \PP^n$ induced by $\cO_X(1)$ is projectively normal.
\end{enumerate}
By the argument in~\cite[(4.6)]{vakil}, it follows that the embedded deformation theory of $X\subseteq \PP^{n-1}$ is smoothly equivalent to the abstract deformation theory of $X$.  Further, ~\cite[Comparison Theorem]{Piene} then implies that the embedded deformation theory of  $X\subseteq \PP^{n-1}$ is equivalent to the $\GG_m$-equivariant embedded deformation theory of the cone of $CX\subseteq \AA^n$.  Hence $\mathcal H_{CX}$ is smoothly equivalent to $R$.

We now let $I_X$ be the ideal defining $X\subseteq \PP^{n-1}$ and we set $I_\Gamma:=I_X+\mathfrak m^m$ for any $m\geq \reg(I_X)+2$.  By Proposition~\ref{prop:muisomorphism}, we conclude that $\mathcal H_{CX}\cong \mathcal H_{\Gamma}$, and hence $\mathcal H_{\Gamma}$ is smoothly equivalent to $R$.  
\end{proof}


\section{Curve singularities with a $\GG_m$-action}\label{sec:examples}
If $C$ is an affine curve $C\subseteq \mathbb A^n$ with a $\GG_m$-action and an isolated singularity, then we use $\Def^{\GG_m}_{C}$ to denote the functor of equivariant deformations of $C$.  For any singularity type $R$, we will construct an equivariant curve $C\subseteq \mathbb A^n$ such that $\mathcal H_C$ lies in the same smooth equivalence class as $R$, and such that the map of functors  $\Def^{\GG_m}_{C|\mathbb A^n}\to \Def^{\GG_m}_{C}$ is smooth.  By definition, $\mathcal H_C$ represents the functor $\Def^{\GG_m}_{C|\mathbb A^n}$, and hence this will prove the corollary.
\begin{proof}[Proof of Corollary~\ref{cor:curves}]
Fix a singularity type $R$.  As in the proof of Theorem~\ref{thm}, we choose a codimension $n-3$ ideal $I_X=(f_1, \dots, f_s)\subseteq S$ such that $\mathcal H_X$ has the same singularity type as $R$.  Set $m\geq \reg(I_X)+2$.  If $\kk$ is an infinite field, then we may define a regular sequence $g_1, g_2$ on $S/I_X$, where $g_i$ is a generic homogeneous form of degreee $m$.  If $\kk$ is finite, then the same holds for some $m\gg \reg(I_X)$.  Let $I_C:=I_X+\langle g_1,g_2\rangle$.

There is an induced map $\mathcal H_X\to \mathcal H_C$, by essentially the same argument as in Theorem~\ref{thm}. The condition of being a regular sequence on $S/I_X$ is open, and hence there is an open set in the vector space $S_m\oplus S_m$ parametrizing the regular sequences.  Thus the choice of $g_1, g_2$ is a smooth choice, and so $\mathcal H_C$ is in the same smooth equivalence class as $\mathcal H_X$.

Since $C$ is affine, we have a formally smooth map of functors $\Def_{C|\mathbb A^n}\to \Def_C$ by \cite[p. 4]{artin}.  Further, by always choosing homogeneous representatives, one can extend Artin's argument to show that the map $\Def^{\GG_m}_{C|\mathbb A^n}\to \Def^{\GG_m}_C$ is formally smooth.  Note that, since the tangent spaces of both functors are finite dimensional, the above map is actually smooth.  As $\mathcal H_C$ is in the same smooth equivalence class as $R$, this completes the proof.
\end{proof}

\section*{Acknowledgements} 
I would like to thank Ravi Vakil and Mauricio Velasco for conversations which inspired this project.  I also thank Dustin Cartwright, David Eisenbud, 
Robin Hartshorne, Brian Osserman,
and Fred van der Wyck for useful conversations.


\end{document}